\documentclass[12pt]{amsart}
\usepackage{amsmath,amssymb,amsfonts,amsthm,amsopn}
\usepackage{graphicx,tikz}
\usepackage[all]{xy}
\usepackage{amsmath} 
\usepackage{multirow, longtable, makecell, caption, array,enumitem}
\usepackage{amssymb}
\usepackage{mathtools}
\usepackage{bookmark}
\usepackage{hyperref}
\calclayout
\newtheorem{theorem}{Theorem}[section]
\newtheorem{lemma}[theorem]{Lemma}

\newtheorem{proposition}[theorem]{Proposition}

\usepackage{color}

\newtheorem{example}[theorem]{Example}
\newtheorem{remark}[theorem]{Remark}

\theoremstyle{definition}

\usepackage{cellspace} %
\def\bR{\mathbb{R}}

\def\bN{\mathbb{N}}

\def\cF{\mathcal{F}}
\def\cS{\mathcal{S}}

\def\rd{\bR^d}

\def\rdd{\bR^{2d}}

\def\la{\langle}
\def\ra{\rangle}
\def\lc{\left(}
\def\rc{\right)}

\def\*b{*_{\bullet}}
\def\w{\mathrm{w}}

\def\S0{S^0_{0,0}}

\def\Bd'{B_{\delta'}}

\def\cBd'{\bar{B}_{\delta'}}

\begin{document}
	
\title[Phase space analysis of the Hermite semigroup]{Phase space analysis of the Hermite semigroup and applications to nonlinear global well-posedness}

\author[D. G. Bhimani]{Divyang G. Bhimani}
\address{TIFR Centre for Applicable Mathematics, Bangalore 560065, India}
\email{divyang@tifrbng.res.in}

\author[R. Manna]{Ramesh Manna}
\address{Department of Mathematics, Indian Institute of Science, 560 012 Bangalore, India}
\email{rameshmanna@iisc.ac.in}

\author[F. Nicola]{Fabio Nicola}
\address{Dipartimento di Scienze Matematiche ``G. L. Lagrange'', Politecnico di Torino, corso Duca degli Abruzzi 24, 10129 Torino, Italy}
\email{fabio.nicola@polito.it}

\author[S. Thangavelu]{Sundaram Thangavelu}
\address{Department of Mathematics, Indian Institute of Science, 560 012 Bangalore, India}
\email{veluma@iisc.ac.in}

\author[S. I. Trapasso]{S. Ivan Trapasso}
\address{Dipartimento di Scienze Matematiche ``G. L. Lagrange'', Politecnico di Torino, corso Duca degli Abruzzi 24, 10129 Torino, Italy}
\email{salvatore.trapasso@polito.it}

%
\begin{abstract}
We study the Hermite operator $H=-\Delta+|x|^2$ in $\rd$ and its fractional powers $H^\beta$, $\beta>0$ in phase space. Namely, we represent functions $f$ via the so-called short-time Fourier, alias Fourier-Wigner  or Bargmann transform $V_g f$ ($g$ being a fixed window function), and we measure their regularity and decay by means of mixed Lebesgue norms in phase space of $V_g f$, that is in terms of membership to modulation spaces $M^{p,q}$, $0< p,q\leq \infty$. We prove the complete range of fixed-time estimates for the semigroup $e^{-tH^\beta}$ when acting on $M^{p,q}$, for every $0< p,q\leq \infty$, exhibiting the optimal global-in-time decay as well as phase-space smoothing. \par

As an application, we establish global well-posedness for the nonlinear heat equation for $H^{\beta}$ with power-type nonlinearity (focusing or defocusing), with small initial data in modulation spaces or in Wiener amalgam spaces. We show that such a global solution exhibits the same optimal decay $e^{-c t}$ as the solution of the corresponding linear equation, where $c=d^\beta$ is the bottom of the spectrum of $H^\beta$. This is in sharp contrast to what happens for the nonlinear focusing heat equation without potential, where blow-up in finite time always occurs for (even small) constant initial data - hence in $M^{\infty,1}$.
\end{abstract}
\maketitle
\section{Introduction and discussion of the results}
The heat semigroup  $e^{t\Delta}$  associated to standard Laplacian has been studied by many authors in PDEs and physics. In fact, the study of the heat semigroup pervades throughout mathematical analysis and physics, being indispensable in many situations. Moreover, there has been increasing interest in applications of the fractional Laplacian to the mathematical modelling of various physical phenomena, see e.g.\ \cite{ng, NL} and the references therein. The Hermite operator (also known as quantum harmonic oscillator) $H=-\Delta+|x|^2$ plays a vital role in quantum mechanics and analysis \cite{folland, st}. Nevertheless, there are only few mathematical papers which deal with fractional powers of Hermite operator $H^{\beta} \ (\beta>0),$ see e.g.\ \cite{bhimani1,EC,st1}. 

 The  spectral decomposition  of $H$ on $\mathbb R^d$ is given by 
\begin{equation}\label{eq spec dec}
H=\sum_{k=0}^{+\infty}(2k+d)P_k,\qquad P_k f=\sum_{|\alpha|=k}\langle f,\Phi_\alpha\rangle\Phi_\alpha,
\end{equation}
where $\langle\cdot,\cdot\rangle$ is the inner product in $L^2(\rd)$ and $\Phi_\alpha$, $\alpha\in \mathbb{N}^{d}$, are the normalised Hermite functions, forming an orthonormal basis for  $L^2(\rd)$. Observe the regularising effect of $P_k$, which takes temperate distributions into Schwartz functions: $P_k:\cS'(\rd)\to\cS(\rd)$. Since 0 is not in the spectrum of $ H $ we can define the fractional powers $ H^\beta$ for any $ \beta \in \mathbb{R} $ by means of the spectral theorem, namely
$$  H^\beta f = \sum_{k=0}^\infty (2k+d)^\beta P_kf.$$
We remark  that $ H^\beta $ is a densely defined unbounded operator for $ \beta > 0 $. We thus define the heat semigroup associated to $H^{\beta}$  $(\beta>0)$ by 
\begin{equation}\label{eq heat semi}
e^{-tH^\beta}f=\sum_{k=0}^{+\infty} e^{-t(2k+d)^\beta} P_kf.
\end{equation}
In this note we study the behaviour of this semigroup on modulation spaces. In order to define these spaces, we recall the   definition  of short-time Fourier transform (STFT - also known as the Bargmann transform \cite{tataru}) of $f\in \cS'(\rd)$ with respect to a fixed window function $ g\in \mathcal{S}(\mathbb R^d)\setminus\{0\}$:
\[
V_g f(x,\xi) \coloneqq \int_{\rd} e^{-i \xi \cdot y}f(y) \, \overline{g(y-x)} \, dy,\quad x,\xi\in\rd. 
\]
We then measure the phase-space content of $f$ by means of mixed Lebesgue (quasi-)norms $L^{p,q}(\rd \times \rd)$ of $V_g f$, leading to the so-called modulation spaces $ M^{p,q}$ \cite{F1,gro book,wang}:
\[
\|f\|_{M^{p,q}}:=\| \|V_g f(x,\xi)\|_{L^p_x}\|_{L^q_\xi},
\]
with $0<p,q\leq\infty$; see Section \ref{PR} for a more general definition involving weights. Heuristically, one can think of a function in $M^{p,q}$ as having the local regularity of a function whose Fourier transform is in $L^q$ and decaying at infinity as a function in $L^p$. We have, in particular, $M^{2,2}=L^2$. Modulation spaces can be equivalently designed as a family of Besov-type spaces with the dyadic geometry in frequency replaced by a decomposition in isometric boxes.  \par
We now state our main result, concerning the action of $e^{-tH^\beta}$ on such spaces. 
\begin{theorem}\label{mainthm} Let $\beta>0$, $0< p_1,p_2,q_1,q_2\leq\infty$ and 
\[
 \frac{1}{\tilde{p}}\coloneqq \max\Big\{\frac{1}{p_2}-\frac{1}{p_1},0\Big\},\ \ \frac{1}{\tilde{q}}\coloneqq \max\Big\{\frac{1}{q_2}-\frac{1}{q_1},0\Big\},\quad \sigma\coloneqq  \frac{d}{2\beta} \Big(\frac{1}{\tilde{p}}+\frac{1}{\tilde{q}}\Big).
 \] Then
\begin{equation}\label{eq mainthm}
\|e^{-tH^\beta} f\|_{M^{p_2,q_2}}\leq C(t)\|f\|_{M^{p_1,q_1}}
\end{equation}
for every $t> 0$, where
\begin{equation}\label{optimalconstant}
C(t)=\begin{cases}
C_0e^{-td^\beta} & t\geq 1\\
C_0t^{-\sigma}& 0<t\leq 1
\end{cases}
\end{equation}
for some $C_0>0$.
\end{theorem}
To the best of our knowledge the result is new even in the case $\beta=1$, $p_1=q_1=p_2=q_2$. \par
Let us give a flavour of the heuristics behind the behaviour of $ C(t).$ First, by testing the above estimate on the ground state of $H$ we see that the decay at infinity is absolutely sharp for every choice of the exponents and also that the exponent of $t$ for $t$ small can never be positive. Indeed, if  $ f = \Phi_0$ then $P_0f=f$ and  $P_k f=0$ for $k\not=0$, so that  $e^{-tH^\beta} f=e^{-td^\beta} f$.

As the modulation spaces $ M^{p,q} $ increase when one of the exponents increases while the other one is kept fixed, it is sufficient to prove the result with $p_2$ replaced by $\min\{p_1,p_2\}$ and similarly for $q_2$, namely we can assume $p_2\leq p_1$, $q_2\leq q_1$. Now, in somewhat sloppy terms, the effect of the map $e^{-t H^\beta}$ in phase space is to damp the content of a function in a way that roughly amounts to the multiplication by the function 
\[
F_t(x,\xi)=e^{-t(|x|^2+|\xi|^2)^\beta}.
\] The operator norm of this pointwise multiplication operator, as a map $L^{p_1,q_1}\to L^{p_2,q_2}$ (mixed-norm Lebesgue spaces in $\rd\times \rd$), can be computed by H\"older's inequality, namely $\|F_t\|_{L^{\tilde{p},\tilde{q}}}=C(t)$, as given in the statement ($0<t\leq1$). Of course the negative exponent of $t$, for $t$ small, is the price to pay for the {\it phase-space smoothing}, when passing from $p_1$ to $p_2<p_1$ or from $q_1$ to $q_2<q_1$. \par
We emphasize that the exponents are just assumed to be positive, without further conditions, similarly to the refined estimates for the heat semigroup in Besov spaces \cite{wang240} (see also \cite[Sec.\ 2.2]{wang}) and in real Hardy spaces \cite{chen}.  \par
We also observe that the same estimates hold  when $M^{p,q}$ is  replaced by the so-called Wiener amalgam space $W^{p,q}\coloneqq \mathcal{F} M^{p,q}$, i.e.\ the image of $M^{p,q}$ under the Fourier transform, endowed with the obvious norm. This follows at once by applying the above estimates to $\mathcal{F}^{-1} f$ and using the fact that $\mathcal{F}^{-1}$ commutes with the spectral projections $P_k$ (because $\mathcal{F}^{-1}\Phi_\alpha=i^{|\alpha|}\Phi_\alpha$) and therefore with $e^{-t H^\beta}$ by \eqref{eq heat semi}.

Similar estimates for the fractional heat semigroup $e^{-t(-\Delta)^\beta}$ were obtained in \cite[Thm.\  3.1]{chen}; see also \cite{wang} for related results. However, our analysis follows a completely different pattern and is necessarily less elementary, because the operator $e^{-tH^\beta}$ is not a Fourier multiplier. Actually, except for the case $\beta=1$, there is not even an explicit integral formula for $e^{-tH^\beta}$ and we rely on the theory of fractional powers and heat kernels of globally elliptic pseudodifferential operators in $\rd$ (see e.g.\ \cite{NR,shubin} for a general account).\par
From the above fixed-time estimates one could deduce Strichartz (space-time) estimates by a standard machinery, via the $TT^\ast$ method or real interpolation (see e.g.\ \cite{tao} and \cite[Thm.\  3.1]{chen} respectively; see also \cite{pierf}). However, in the case of the Hermite semigroup one should be able to obtain a broader range of space-time estimates beyond those derived by the fixed-time estimates, see for instance \cite[Cor.\ 2.1]{wang} for the heat semigroup. Hence we prefer to postpone a systematic study of Strichartz estimates, including some interesting related topics - in fact, it seems that the techniques of this paper could be successfully applied to obtain some new $L^p$ estimates as well. Instead, here we focus on some direct applications of Theorem \ref{mainthm} to the heat equation for $H^\beta$ with a nonlinearity of power type, providing some results which are definitely a consequence of the trapping effect of the quadratic potential in $H$ and do not hold for the corresponding heat equation without potential. 

Specifically, we  consider the Cauchy problem for the nonlinear heat equation\footnote{The subsequent arguments and results can be trivially modified for other algebraic nonlinearities such as $\lambda u^k$,
	$\lambda\in\mathbb{C}$, $k\in\mathbb{N}$, $k\geq 2$.} associated to $H^{\beta}:$
\begin{equation} \label{HE}
\begin{cases} \partial_t u+H^\beta u=\lambda |u|^{2k} u\\
u(0,x)=u_0(x)
\end{cases}
\end{equation}
with $(t,x)\in(0,+\infty)\times\rd$, where $k\in \mathbb{N}$, $k\not=0$, $\lambda\in\mathbb{C}$ and $\beta>0$. \par
As an application of Theorem \ref{mainthm}, we   establish global well-posedness for \eqref{HE}. Specifically, we  have following result.
\begin{theorem}\label{mt2}	
Let $ 1\leq p \leq \infty$, $1\leq q \leq \frac{2k+1}{2k}$ and 
\[
\frac{1}{q}+\frac{\beta}{kd}>1.
\]
  Define the subspace $X\subset L^\infty([0,+\infty),M^{p,q})$ of elements $u$ satisfying  
\[\|u\|_{X}\coloneqq  \left\| e^{td^{\beta}}  \|u(t, \cdot)\|_{M^{p,q}} \right\|_{L_t^{\infty}([0, \infty))}< \infty .\]
\begin{itemize}
\item[(a)] There exists $\varepsilon>0$ such that for any $u_0\in M^{p,q}$  satisfying  $\|u_0\|_{M^{p,q}} \leq  \varepsilon,$  the problem \eqref{HE} has
	 a unique global solution	 
	 \[u \in L^{\infty}([0, \infty), M^{p,q}). \]
\item[(b)] We have $u \in C([0,\infty), M^{p,q})$  if $p<\infty$. \par\noindent
\item[(c)] If $\varepsilon$ is small enough the above solution enjoys exponential decay in time; specifically,   $u \in X.$	 \par\noindent
\item[(d)] The same results hold with $M^{p,q}$ replaced by $W^{q,p}$, for the same range of exponents. 
\end{itemize}
\end{theorem} 
The following remarks are in order (see Example \ref{ex} for details). 

\begin{remark}\label{fe}
Let $f_\alpha(x)=|x|^{-\alpha}$, $0\leq\alpha<d$, $x\in\rd$.
\begin{itemize}
\item[(i)]
The hypothesis  $1\leq q\leq \frac{2k+1}{2k}$ in Theorem \ref{mt2} is natural.  Indeed, $f_\alpha\in M^{\infty,q}\subset W^{q,\infty}$, $q=\frac{2k+1}{2k}$, if $0\leq \alpha<\frac{d}{2k+1}$. But in the limiting case $\alpha=\frac{d}{2k+1}$ we 
have $|f_\alpha(x)|^{2k} f_\alpha(x)=|x|^{-d}$, which is not even locally integrable. Notice that when $\beta>\frac{kd}{2k+1}$, under the assumption $1\leq q\leq\frac{2k+1}{2k}$ the additional condition $\frac{1}{q}+\frac{\beta}{kd}>1$ is automatically satisfied.
\item[(ii)] To give a flavour of the the type of singularities and oscillations at infinity admitted for the initial data, we observe that, for example, one can take as initial datum $u_0(x)=\varepsilon f_\alpha(x)(1+c\cos |\xi|^2)$, $c\in\mathbb{C}$, $0\leq\alpha<\min\{d/(2k+1),\beta/k\}$, with $\varepsilon$ small enough. Also, if $\chi$ is any smooth function with compact support in $\rd$ and $\Lambda$ is any lattice in $\rd$, one can consider $u_0(x)=\varepsilon\sum_{\mu \in\Lambda} f_\alpha(x-\mu)\chi(x-\mu)$, with the same restrictions on $\alpha$ and $\varepsilon$. Observe that of course $f_\alpha\not\in L^p$ for every $p$ (if $\alpha\not=0$), thus Theorem \ref{mt2} reveals that we can control initial data beyond $L^p$ (there is an enormous literature on nonlinear heat equations  with Cauchy data in $L^p$; see e.g.\ \cite{ms, vas} and the references therein).
\item[(iii)] The results in Theorem \ref{mt2} look interesting because they are in sharp contrast to what happens for the standard heat equation with the above nonlinearity and $\lambda=1$ (focusing case), where for real constant initial data (hence in $M^{\infty,1}$), even small, one has always blow-up in finite time for every $k\not=0$, as one sees at once by solving the ordinary differential equation $u_t=u^{k+1}$, $u$ real (again the literature in this connection is large; see  \cite{vas} for a comprehensive survey).
\item[(iv)] In Theorem \ref{mt2} we suppose $p,q\geq 1$. In fact, we are interested in well-posedness in the lowest regularity/biggest spaces; moreover, the proof relies on the Minkowski integral inequality. However, the problem of the persistence of regularity in quasi-Banach spaces, as well as weighted variants of the above results, although not of primary interest, could be worth investigating.
\end{itemize}

\end{remark}

The proof of Theorem \ref{mt2} can be adapted (in fact simplified) to prove local well-posedness for \eqref{HE} in the same spaces without any smallness assumption on the initial data. We leave the details to the interested reader and we limit ourselves to briefly state the result as follows 
(with the necessary clarifications in the case $p=\infty$, as in Theorem \ref{mt2}).
\begin{theorem}\label{LW}
Let $ 1\leq p \leq \infty$,  $ 1\leq q \leq \frac{2k+1}{2k}$ and $\frac{1}{q}+\frac{\beta}{kd}>1.$  Then \eqref{HE}  is locally well-posed in $M^{p,q}$ and $W^{q, p}$.
\end{theorem}

We observe that this result implies, in particular, local well-posedness in $M^{p,1}$, $1\leq p\leq\infty$. Actually this special case, when $0<\beta\leq 1$, follows directly from \cite[Thm.\  1.1]{nicola} - it was also re-obtained in \cite{EC} (in fact, if $\beta\leq 1$, the operator $H^\beta$ is a pseudodifferentual operator with a real-valued Weyl symbol, bounded from below and with bounded derivatives of order $\geq 2$, therefore satisfying the assumptions in \cite[Thm.\ 1.1]{nicola}; see Proposition \ref{prop shubin} below). However  Theorem \ref{LW} applies to every $\beta>0$  and for a range of $q>1$, allowing more singular initial data (in contrast, $M^{\infty,1}$ contains only continuous functions).\par
Modulation spaces have been widely applied in the study of nonlinear PDEs. The local and global well-posedness  for the  heat equation associated to   Laplacian  in weighted modulation spaces  goes back to the  work of Iwabuchi  \cite{ib}. In \cite{FD}, authors have proved ill-posedness for the fractional heat equation  and in \cite[Thm.\ 1.1]{bhimani1}  finite time blow-up  has been established in some modulation spaces.  On the other hand, Bhimani et al.\ proved in \cite{bhimani3}   global well-posedness for the Hartree-Fock equations associated to harmonic oscillator in some modulation spaces; see also \cite{RM}.  There is a large literature dealing with the analysis of PDEs on modulation spaces; we refer to the surveys \cite{BOP,RSW} and the monograph \cite{wang}, and the references therein (see also \cite{BO,CNR}); we also mention the article \cite{kt} for results on the Hermite operator obtained using phase-space methods. However, the study of nonlinear {\it global} well-posedness in modulation and Wiener amalgam spaces is very limited when the corresponding linear propagator is not a Fourier multiplier and, in fact, new interesting phenomena can occur, as observed in Remark \ref{fe} {\it (iii)}.
\par\bigskip
In short, the paper is organised as follows. We collect some background material on modulation spaces and a number of preliminary results in Section 2. Section 3 is devoted to the proof of Theorem \ref{mainthm}, while in Section 4 we provide the proof of Theorem \ref{mt2}.


\section{Preliminary results}\label{PR}
We write $|x|^2 = x\cdot x$ for $x \in \rd$, where $x\cdot y$ is the inner product on $\rd$. We denote by $\cS(\rd)$ the Schwartz class of rapidly decaying smooth functions on $\rd$ and by $\cS'(\rd)$ the space of temperate distributions. The bracket $\la f,g\ra$ stands for the inner product of $f,g \in L^2(\rd)$ as well as for the action of $f \in \cS'(\rd)$ on $g \in \cS(\rd)$; in both cases we assume it to be conjugate-linear in the second entry. 

Given $x,\xi \in \rd$, the translation operator $T_x$ and the modulation operator $M_{\xi}$ are defined as 
\[ T_x f(y) \coloneqq f(y-x),\quad M_\xi f(y) \coloneqq e^{i \xi \cdot y} f(y), \quad f \in \cS(\rd), \]
and can be extended to temperate distributions by duality. The composition $\pi(z) = M_{\xi}T_x$, $z = (x,\xi) \in \rdd$, is referred to as a time-frequency shift. 

The short-time Fourier transform (STFT) of $f \in \cS'(\rd)$ with respect to a window function $g \in \cS(\rd)\setminus\{0\}$ is defined by
\[ V_g f(x,\xi) \coloneqq \la f,\pi(x,\xi)g \ra = \int_{\rd} e^{-i \xi \cdot y}f(y)\overline{g(y-x)}dy. \]

Modulation spaces were introduced by Feichtinger \cite{F1} in the '80s. They consist of functions enjoying suitable summability/decay conditions on the phase-space side. Consider a weight $m(x,\xi)$ in phase space, i.e.\ a continuous and strictly positive function in $\rd\times\rd$ with at most polynomial growth - in fact, we will often use the polynomial weight $v_s(x,\xi) \coloneqq (1+|x|+|\xi|)^s$, $(x,\xi) \in \rd\times\rd$, $s\in\mathbb{R}$ \footnote{Actually we need a further technical assumption which will be always verified in the following - hence the reader could ignore this issue - namely $m$ is required to be $v_s$-moderate for some $s\geq0$; see \cite{gro book}.}. Let $0< p, q \le \infty$ and $g \in \cS(\rd)\setminus\{0\}$; the modulation space $M^{p,q}_m(\rd)$ is the set of all $f\in \cS'(\rd)$ such that
\begin{equation}\label{modsp norm} \| f \|_{M^{p,q}_{m}} \coloneqq \| V_g f \, m \|_{L^{p,q}} = \lc \int_{\rd} \lc \int_{\rd} |V_g f(x,\xi) m(x,\xi)|^p dx \rc^{q/p} d\xi \rc^{1/q} < \infty, \end{equation}
with obvious modifications in the case where $p=  \infty$ or $q=\infty$. When $s = 0$ we simply write $M^{p,q}$. Here we used the notation $L^{p,q}(\rd \times \rd)$ for the mixed-norm Lebesgue spaces, with (quasi-)norm $\|F(x,\xi)\|_{L^{p,q}}= \|\|F(x,\xi)\|_{L^p_x}\|_{L^q_\xi}$.

It turns out that modulation spaces are quasi-Banach spaces (Banach spaces if $p,q\geq 1$) whose definition is independent of the choice of the window function $g$ - in the sense that different choices of the window provide equivalent norms; see \cite{CR,galperin,gro book} for proofs and further details. Here we observe that they have a number of relations with standard function spaces of harmonic analysis, the most notable being that $M^{2,2}(\rd) = L^2(\rd)$, and the so-called Shubin-Sobolev, alias Hermite-Sobolev, spaces $Q^s$, $s\in\mathbb{R}$ \cite{shubin}, which can be defined \cite[Thm.\ 2.1]{gramchev} as the space of $f\in\cS'(\rd)$ such that 
\begin{equation}\label{eq char shusob}
\|f\|^2_{Q^s}\coloneqq \|H^{s/2}f\|^2_{L^2}=\sum_{k=0}^{+\infty} ||P_k f||^2_{L^2}(2k+d)^{s}<\infty.
\end{equation}
It is well known that, for every $0<p,q\leq\infty$, if $s$ is large enough,
\begin{equation}\label{eq emb shusob}
Q^s\hookrightarrow M^{p,q}\hookrightarrow M^\infty\hookrightarrow Q^{-s}\end{equation} (this follows easily from the characterisation $Q^s=M^{2,2}_{v_s}$ \cite[Lem.\ 4.4.19]{CR}, H\"older's inequality and the inclusion relations \cite[Thm.\ 2.4.17]{CR}). 

We recall that reversing the order of integration in \eqref{modsp norm}, namely
\begin{equation}\label{was norm} \| f \|_{W^{p,q}} \coloneqq \lc \int_{\rd} \lc \int_{\rd} |V_g f(x,\xi)|^p d\xi \rc^{q/p} dx \rc^{1/q}, \end{equation} (here we take $m=1$ for simplicity) gives rise to a norm that characterizes the so-called Wiener amalgam spaces $W^{p,q}$. In fact, they are strictly related to modulation spaces via the Fourier transform, since $W^{p,q} = \cF M^{p,q}$.  We stress that such spaces can be equivalently characterized by decomposition methods as Wiener amalgams with local component $\cF L^p$ and global component $L^q$, that is $W^{p,q} = W(\cF L^p,L^q)(\rd)$ - see \cite{CR,F0} for further details.

\begin{example}\label{ex}
\begin{itemize}
\item[(i)]
Let $f_\alpha(x)=|x|^{-\alpha}$, $0<\alpha<d$ (cf.\ Remark \ref{fe}). Let us show that the function $f_\alpha(x)$ is in $M^{p,q}(\mathbb R^d)$ for  $p>d/\alpha$ and $q>d/(d-\alpha)$. We just sketch the proof, leaving the details to the interested reader. \par One can estimate separately the STFT of $\chi f_\alpha$ and $(1-\chi) f_\alpha$, where $\chi$ is smooth with compact support in $\mathbb{R}^d$, $\chi=1$ in a neighborhood of the origin, taking a window $g$ with compact support. Then $V_g (\chi f)(x,\xi)=0$ if $x$ is large enough and for $x$ in a compact subset one uses $|V_g (\chi f)(x,\xi)|\lesssim |\hat{f_\alpha}|\ast_\xi |\hat{\chi}| \ast_\xi |\hat{\overline{g}}|(\xi)=:F(\xi)$. Indeed, since $\hat{f_\alpha}=c_\alpha f_{d-\alpha}$ for some $c_\alpha\in\mathbb{R}$, by invoking the Hardy-Littlewood-Sobolev inequality we have $F\in L^q(\mathbb R^d)$ for $q> \frac{d}{d-\alpha}$. On the other hand,  $(1-\chi)f_\alpha$ can be estimated using the embedding $L^p_k \hookrightarrow M^{p,q}$ which holds for $1 < p < \infty$ and $k \in \bN$ large enough, where $L^p_k$ denotes the space of $L^p$ functions with $k$ distribution derivatives in $L^p$ \cite{kob}. 
\item[(ii)]
Since, as already observed, $\hat{f_\alpha}=c_\alpha f_{d-\alpha}$, we have $f_\alpha\in W^{q,p}(\mathbb R^d)$ for the same range of exponents $p,q$.
\item[(iii)] Note that $\cos |x| \in M^{\infty,1}$ \cite[Cor.\ 15]{bgor unimod} and by the algebra property (Proposition \ref{gap}) we have $M^{p,q}\cdot M^{\infty,1}\subset M^{p,q}$, $1\le p,q \le \infty$. Thus, if $f\in M^{p,q}$ then $ f(x) \cos |x|\in M^{p,q}$. Similarly, using that $\cos |x|^2 \in W^{1,\infty}$ \cite[Thm.\ 14]{bgor unimod} and  $W^{q,p}\cdot W^{1,\infty}\subset W^{q,p}$, $1 \le p,q \le \infty$, we see that if $f$ belongs to $W^{q,p}$ then $ f(x) \cos |x|^2$ belongs to $W^{q,p}$ too. 
\item[(iv)] Constant functions are in $M^{\infty,1}\subset W^{1,\infty}$. 
\item[(v)] If $\chi$ is any smooth function with compact support in $\rd$ and $\Lambda\subset\mathbb{R}^d$ is any lattice in $\rd$, $q>d/(d-\alpha)$, then $\sum_{\mu \in\Lambda} f_\alpha(x-\mu)\chi(x-\mu)$ belongs to $W^{q,\infty}$, as one verifies easily.
\end{itemize}

\end{example}

Modulation spaces can be used both as symbol classes as well as environment where to study boundedness of pseudodifferential operators \cite{CR}, i.e.\ operators formally given by  
\[
a^\w f(x)=(2\pi)^{-d}\int_{\rdd} e^{i\xi\cdot (x-y)} a\Big(\frac{x+y}{2},\xi\Big)f(y)\, dy d\xi
\]
where $a(x,\xi)$ is a function in phase space - the Weyl symbol of the operator $a^\w$. In fact the above integral, suitably interpreted in a weak sense, gives rise to a continuous operator $a^\w:\cS(\rd)\to \cS'(\rd)$ for any distribution symbol $a \in \cS'(\rdd)$. However, in the following we will only consider smooth symbols satisfying some growth conditions at infinity. Indeed, the relevant symbol classes in this paper are given by the so-called Shubin classes \cite{helffer,NR,shubin}: for $s\in\mathbb{R}$ we define $\Gamma^s$ as the space of functions $a\in C^\infty(\rdd)$ such that 
\[
|\partial^\alpha a(z)|\leq C_\alpha (1+|z|)^{s-|\alpha|}\qquad z\in\rdd
\]
for every $\alpha\in \mathbb{N}^{2d}$. This space becomes a Fr\'echet space when endowed with the obvious seminorms. \par


We state a generalized version of the Calder\'on-Vaillancourt theorem that will be used below.

\begin{theorem}\label{continuity}
Let $a\in \Gamma^{-s}$, $s\in\mathbb{R}$, and $0< p,q\leq \infty$. Then $a^\w: M^{p,q}\to M^{p,q}_{v_s}$ continuously, with operator norm depending only on a finite number of seminorms of $a$ in $\Gamma^{-s}$. 
\end{theorem}
\begin{proof}
The desired continuity result follows from \cite[Thm.\ 3.1]{toft} specialised to the case of weights $\omega_1(x,\xi)= 1$,  $\omega_2(x,\xi)=(1+|x|+|\xi|)^s$ ($x,\xi\in\rd$), $\omega_0(z,w)= (1+|z|+|w|)^s$ ($z,w\in\rdd$), which gives the continuity of $a^\w: M^{p,q}\to M^{p,q}_{v_s}$ if $a\in M^{\infty,r}_{\omega_0}(\rdd)$ for $r\leq\min\{1,p,q\}$. \par
On the other hand, if a symbol $a$ satisfies the estimates
\[
|\partial^\alpha a(z)|  \leq C_\alpha (1+|z|)^{-s}\qquad z\in\rdd
\]
(no additional decay is needed for the derivatives), then it belongs to $M^{\infty,r}_{\omega_0}(\rdd)$ for any $r>0$. In fact, writing $e^{-iw\cdot y} =(1+|w|^2)^{-N}(1-\Delta_y)^N e^{-iw\cdot y}$ $(z,w,y\in\rdd)$ in the formula for the short-time Fourier transform of $a$, and repeated integration by parts yield
\[
|V_g a(z,w)|\leq C_N (1+|w|)^{-2N}(1+|z|)^{-s}\qquad z,w\in\rdd
\]
for every $N\in\bN$ ($g\in\cS(\rdd)$), which easily gives the claim. 
\end{proof}

To conclude this section we recall a result about complex powers of pseudodifferential operators. 

\begin{proposition}\label{prop shubin}
	Let $\beta >0$. The fractional Hermite operator $H^{\beta} = (-\Delta+|x|^2)^{\beta}$ is a pseudodifferential operator with Weyl symbol $a_{\beta} \in \Gamma^{2\beta}$. More precisely we have
	\begin{equation}\label{eq prop shubin}
	a_\beta(x,\xi) = (|x|^2+|\xi|^2)^\beta + r(x,\xi), \quad |x|+|\xi|\ge 1,
	\end{equation}  where $r \in \Gamma^{2\beta-2}$.
\end{proposition}
\begin{proof}
	The result follows from the machinery of complex powers applied to the operator $H$ (cf.\ \cite[Thm.\ 1.11.1, Thm.\ 1.11.2]{helffer}). Indeed $H$ is a positive operator with Weyl symbol $|x|^2+|\xi|^2$, which is positive globally elliptic - in the sense that $|x|^2+|\xi|^2 \ge C(1+|x|+|\xi|)^2$ for $|x|+|\xi|$ sufficiently large and some $C>0$. The desired result then follows e.g.\ from \cite[Thm.\ 4.3.6]{NR} specialized to the symbol class $S(M,\Phi,\Psi)=\Gamma^m$, namely with $M(x,\xi)=(1+|x|+|\xi|)^{m}$, $\Phi(x,\xi)=\Psi(x,\xi)=1+|x|+|\xi|$; the so-called Planck function $h(x,\xi)=\Phi(x,\xi)^{-1}\Psi(x,\xi)^{-1}=(1+|x|+|\xi|)^{-2}$, which gives the gain in the asymptotic expansions within this symbol class, is responsible for the gain in decay of the remainder $r(x,\xi)$ compared with the ``principal symbol'' $(|x|^2+|\xi|^2)^\beta$. 
\end{proof}

\section{Proof of the main result}
\begin{proof}[Proof of Theorem \ref{mainthm}]
We prove the desired estimate separately in the regimes $t\geq 1$ and $0< t\leq 1$. \par\medskip
{\bf Case $t\geq 1$.}
It is sufficient to prove the following estimates:
\begin{equation}\label{uno}
\|P_k\|_{M^{p_1,q_1}\to M^{p_2,q_2}}\leq C_0(2k+d)^{s},\qquad k\in\mathbb{N}
\end{equation}
for some $s\geq 0$ and $C_0>0$, and
\begin{equation}\label{due}
\sum_{k=0}^{+\infty} e^{-t(2k+d)^\beta} (2k+d)^{s} \leq C_1 e^{-td^\beta},\qquad t\geq 1
\end{equation} for some $C_1>0$.
Let us prove \eqref{uno}. 
As a consequence of the characterization \eqref{eq char shusob} and the embeddings in \eqref{eq emb shusob}, for $s$ large enough we have
\[
\|P_k\|_{M^{p_1,q_1}\to M^{p_2,q_2}}\leq C_0 \|P_k\|_{Q^{-s}\to Q^{s}}=C_0(2k+d)^{s}.
 \]
Let us now prove \eqref{due}. Since the sequence $\mathbb{N}\ni k\mapsto e^{-t(2k+d)^\beta}(2k+d)^{s}$ is decreasing for, say, $k\geq k_0$, we estimate separately
\[
\sum_{k=0}^{k_0} e^{-t(2k+d)^\beta} (2k+d)^{s} \leq C_2 e^{-td^\beta}
\]
for some $C_2>0$, and 
\begin{align*}
\sum_{k=k_0+1}^{+\infty} e^{-t(2k+d)^\beta}(2k+d)^{s}&\leq \int_{k_0}^{+\infty} e^{-t(2x+d)^\beta}(2x+d)^{s}\, dx\\
&\leq  \int_{0}^{+\infty} e^{-t(2x+d)^\beta}(2x+d)^{s}\, dx\\
 &=\frac{e^{-td^\beta}}{2\beta}\int_0^{+\infty} e^{-ty} (y+d^\beta)^{-1+1/\beta+s/\beta}\, dy,
\end{align*}
where we applied the change of variable $(2x+d)^\beta=y+d^\beta$. The latter integral is decreasing in $t$, so that its value for $t\geq 1$ is not larger than that corresponding to $t=1$, and the proof is concluded.

\par\bigskip
{\bf Case $0< t\leq 1$.}
First of all we observe that, by the already mentioned inclusion relations of modulation spaces, we can limit ourselves to prove the desired result with $p_2$ replaced by $\min\{p_1,p_2\}$ and $q_2$ replaced by $\min\{q_1,q_2\}$. Hence from now on $p_2\leq p_1$ and $q_2\leq q_1$. \par
Recall from Proposition \ref{prop shubin} that $H^\beta$ is a pseudodifferential operator with a real-valued Weyl symbol $a_\beta \in \Gamma^{2\beta}$. Moreover, the machinery of the heat kernel of pseudodifferential operators applies (see e.g.\ \cite[Thm.\ 4.5.1]{NR}, the global ellipticity assumption being satisfied in view of the structure \eqref{eq prop shubin} of $a_\beta$), and the associated heat semigroup $e^{-tH^\beta}$ is therefore a pseudodifferential operator with Weyl symbol $b_t(x,\xi)$, depending on the parameter $t$, such that, for every $N\geq 0$, the symbol $t^Nb_t$ belongs to a bounded subset of $\Gamma^{-2\beta N}$ when $t$ stays in any compact subset of $[0,+\infty)$.\par Now, fix $N\in\mathbb{N}$ such that $2\beta N> 2d$; we can apply Theorem \ref{continuity} to the symbols $b_t$ and $t^N b_t$ and we obtain, with $m(x,\xi)=v_{2\beta}(x,\xi)=(1+|x|+|\xi|)^{2\beta}$,
\[
\| (1+t^Nm^N) V_g(e^{-tH^\beta} f)\|_{L^{p_1,q_1}}\leq C \|f\|_{M^{p_1,q_1}}
\]
for a constant $C$ independent of $t\in (0,1]$. Hence it is sufficient to prove that 
\[
\|F\|_{L^{p_2,q_2}}\leq C(t) \| (1+t^N m^N) F\|_{L^{p_1,q_1}}
\]
with $C(t)$ as in the statement and for every measurable function $F(x,\xi)$. This follows by H\"older's inequality since, under our assumption, $1/p_2=1/p_1+1/\tilde{p}$ and  $1/q_2=1/q_1+1/\tilde{q}$ and
\[
\|(1+t^N m(x,\xi)^N)^{-1}\|_{L^{\tilde{p},\tilde{q}}}\leq \|(1+t^N(|x|+|\xi|)^{2\beta N})^{-1}\|_{L^{\tilde{p},\tilde{q}}}=C(t)
\]
as one sees by a linear change of variable. This proves the desired estimate in the regime $t\in (0,1]$.
  \end{proof}
\begin{remark}
We highlight that the estimate \eqref{uno} for $p_1=p_2=q_1=q_2$ can be improved as follows: 
\begin{equation}\label{unobis}
\|P_k\|_{M^{p,p}\to M^{p,p}}\leq 1,\qquad k\in\mathbb{N},\ p\geq 1,
\end{equation}
for a suitable choice of the window implicit in the definition of the $M^{p,p}$-norm.
In fact, for $k\in\mathbb{N}$ we can write
\begin{align*}
P_k&=(2\pi)^{-1}\sum_{\ell =0}^{+\infty}  \left( \int_0^{2\pi} e^{-i\theta(2\ell+d)} e^{i\theta(2k+d)}\,d\theta \right)  \, P_\ell\\
&= (2\pi)^{-1} \int_0^{2\pi} e^{-i\theta H} e^{i\theta(2k+d)}\, d\theta,
\end{align*}
where the interchange of the sum and the integral is justified by Fubini Theorem, after interpreting the above equalities in weak sense and using that for $f,g\in\cS(\rd)$ one has 
\[
\sum_{\ell=0}^{+\infty} |\langle P_\ell  f,g\rangle|= \sum_{\ell=0}^{+\infty} |\langle P_\ell  f,P_\ell g\rangle|\leq \|f\|_{L^2} \|g\|_{L^2}.
\]
The estimate \eqref{unobis} then follows at once from the fact that 
\[
\|e^{-i\theta H}\|_{M^{p,p}\to M^{p,p}}=1,\qquad \theta\in\mathbb{R},
\]
for a suitable choice of the window implicit in the definition of the $M^{p,p}$-norm; see the proof of \cite[Thm.\ 1.7]{bhimani2}. \par
This argument fails for $M^{p,q}$ with $p\not=q$, because in that case $e^{-i\theta H}$ is not bounded on $M^{p,q}$ (except for special values of $\theta$).\par
\end{remark}
\section{The nonlinear heat equation for $H^\beta$}
In this section we prove Theorem \ref{mt2}. We recall the following results \cite{F1,kb,wang dec}.

\begin{proposition}[Algebra property]\label{gap}  Let $m\in \mathbb N$, $m\geq1$.  Assume that $\sum_{i=1}^{m} \frac{1}{p_i}= \frac{1}{p_0}$, $\sum_{i=1}^{m} \frac{1}{q_i}= m-1+ \frac{1}{q_0}$ with $0<p_i\leq \infty, 1\leq q_i \leq \infty$  for $1\leq i \leq m.$ Then, for some $C>0$,
\[  \left \| \prod_{i=1}^{m} f_i \right\|_{M^{p_0, q_0}} \leq C  \prod_{i=1}^{m} \|f_i\|_{M^{p_i, q_i}}.\]
\end{proposition}
\begin{lemma} \label{Algebra}
		The multi-linear estimates
		\[
		\||f|^{2k}f\|_{M^{p,r}}\leq C \|f\|^{2k+1}_{M^{p,q}}		\]	
		hold for $1\leq p,q,r \leq \infty,  k\in \mathbb N$, $\frac{2k+1}{q}=\frac{1}{r}+2k$.
	\end{lemma}	
\begin{proof} From Proposition \ref{gap} we have
	\[ \||f|^{2k}f\|_{M^{p_0,r}} \leq C \|f\|_{M^{p,q}}^{2k+1}, \] with $\frac{2k+1}{p} = \frac{1}{p_0}$ and the conclusion follows from the embedding $M^{p_0,r} \hookrightarrow M^{p,r}$ ($p_0\leq p$).
\end{proof}

\begin{proof}[\textbf{Proof of Theorem \ref{mt2}}]
		{\it (a)} We study \eqref{HE} directly in integral form (Duhamel's principle), namely
				\begin{eqnarray} \label{HE1}
		u(t)=S(t)u_0+ \lambda \int_0^t S(t-\tau) \, [|u(\tau)|^{2k}u(\tau)] \, d \tau=:\mathcal{J}(u)
		\end{eqnarray}
	where $S(t)=e^{-t H^{\beta}}.$\par Let $p,q$ be as in the statement. 
By Theorem \ref{mainthm} we have, for $t\geq0$, 
\begin{equation}\label{14bis}
\|S(t) f\|_{M^{p,q}}\leq C_1 \|f\|_{M^{p,q}}
\end{equation} 
for some $C_1>0$ and, for $r\geq q$,
		\begin{equation*}\label{eq2}
		\|S(t) f\|_{M^{p,q}}\leq C(t)\|f\|_{M^{p,r}},
		\end{equation*}
		where
		\begin{equation*}
		C(t)=\begin{cases}
		C_0 e^{-td^\beta} & t\geq 1\\
		C_0 t^{-\sigma}& 0<t\leq 1
		\end{cases}
		\end{equation*}
		for some $C_0>0,$ with $\sigma=\frac{d}{2\beta}\left(\frac{1}{q}-\frac{1}{r}\right).$  Since $1\leq q \leq \frac{2k+1}{2k}$ we can  choose $r\in [1, \infty]$ such that $\frac{2k+1}{q}=\frac{1}{r}+2k$. 		
		 Since $ \frac{1}{q}+ \frac{\beta}{kd}>1,$ we have $\sigma<1$.
By Minkowski's inequality for integrals and Lemma \ref{Algebra}, we obtain, for some constants $C_2,C_3>0$,
		\begin{eqnarray*}\label{eq3}
		&&\hspace{-.5in}\left\|\int_0^t S(t-\tau) \, [|u(\tau)|^{2k} u(\tau)] \, d \tau \right\|_{ M^{p,q}} \nonumber \\
		&& \leq \int_{0}^t \left\| S(t-\tau) \, [|u(\tau)|^{2k} u(\tau)] \right\|_{M^{p,q}} \, d \tau \nonumber \\
		&& \leq \int_{0}^t \, C(t-\tau) \, \| |u(\tau)|^{2k}u(\tau) \|_{M^{p,r}} \, d \tau \nonumber \\
		&& \leq C_2  \int_0^t C(t-\tau) \|u(\tau)\|_{M^{p,q}}^{2k+1} d\tau \nonumber \\
		&&  \leq  C_2 \|u\|^{2k+1}_{L^{\infty}([0,t], M^{p,q})} \int_0^t C(s) ds \leq C_3 \, \|u\|^{2k+1}_{L^{\infty}([0,+\infty), M^{p,q})}. 
		\end{eqnarray*}
Combining this inequality with \eqref{14bis}, we have
		\begin{eqnarray} \label{eq4}
		\|\mathcal{J}(u)\|_{L^{\infty}([0,+\infty), M^{p,q})} &\leq& C_4 \big(  \|u_0\|_{M^{p,q}}+  \|u\|^{2k+1}_{L^{\infty}([0,+\infty), M^{p,q})}\big)
		\end{eqnarray}
		for some constant $C_4>0$.
For $\varepsilon>0,$ put \[
B_{\varepsilon}=\{u \in L^{\infty}([0,+\infty), M^{p,q}): \|u\|_{L^{\infty}([0,\infty), M^{p,q})} \leq \varepsilon \},\]
which is the closed ball of radius $\varepsilon,$ centred at the origin in $L^{\infty}([0,+\infty), M^{p,q}).$ Next, we show that the mapping $\mathcal{J}$ maps $B_{\varepsilon}$ into itself for suitable choice of $\varepsilon$. 
		Now, if we assume $\|u_0\|_{M^{p,q}} \leq \frac{\varepsilon}{2 \, C_4}$ then 
		from \eqref{eq4} we obtain for $u \in B_{\varepsilon}$
		\begin{eqnarray*}
			\|\mathcal{J}(u)\|_{L^{\infty}([0,+\infty), M^{p,q})} \leq \frac{\varepsilon}{2}+ C_4 \varepsilon^{2k+1}.
		\end{eqnarray*}
		Since $k>0$, we can choose $\varepsilon$ such that $\varepsilon^{2k} \leq \frac{1}{2 \, C_4}$ and as a consequence we have
		\begin{eqnarray*}
			\|\mathcal{J}(u)\|_{L^{\infty}([0,+\infty), M^{p,q})} \leq \frac{\varepsilon}{2} + \frac{\varepsilon}{2}=\varepsilon,
		\end{eqnarray*}
		that is, $\mathcal{J}(u) \in B_{\varepsilon}.$ 
Noticing the identity
\[|u|^{2k}u-|v|^{2k}v= (u-v) |u|^{2k} + v(|u|^{2k}-|v|^{2k}) \]		
	 and  exploiting similar arguments as before, we obtain 
		\begin{eqnarray*}
			\|\mathcal{J}(u)- \mathcal{J}(v)\|_{L^{\infty}([0,+\infty), M^{p,q})} \leq  \frac{1}{2} \, \| u-v\|_{L^{\infty}([0,+\infty), M^{p,q})},
		\end{eqnarray*}
possibly by taking $\varepsilon$ smaller. Therefore, using Banach's contraction principle, we conclude that $\mathcal{J}$ has a unique fixed point in $B_{\varepsilon}$ which is the solution of \eqref{HE1}.\par\bigskip

{\it (b)} Let us now show that when $p<\infty$ ($q<\infty$ because of the assumption $q\leq \frac{2k+1}{2k}$) the unique solution in $L^\infty([0,+\infty),M^{p,q})$ in fact is continuous in $t$, i.e.\ belongs to $C([0,+\infty),M^{p,q}$).
It is sufficient to repeat the above contraction argument with the space $C([0,+\infty),M^{p,q})$ in place of  $L^\infty([0,+\infty),M^{p,q})$, provided that the semigroup $S(t)$ is strongly continuous on $M^{p,q}$. To this end, observe that by \eqref{14bis} it is sufficient to prove that the map $t\mapsto S(t)f$ is continuous with values in $M^{p,q}$ for every $f$ in some dense subset of $M^{p,q}$. \par
Then, let us take $f\in \mathcal{S}(\mathbb{R}^d)$. We know that the semigroup $S(t)$ is strongly continuous on $L^2$. Hence for every $k\in\mathbb{N}$, the map $t\mapsto S(t) H^k  f=H^k S(t) f$ is continuous with values in $L^2$. But the seminorms $p_k(f)\coloneqq \|H^k f\|_{L^2}$, $k\in\mathbb{N}$, define an equivalent family of seminorms for $\mathcal{S}(\mathbb{R}^d)$, see  \cite{NR}. Hence the map $t\mapsto S(t) f$ is continuous with values in $\mathcal{S}(\mathbb{R}^d)$ and a fortiori when regarded as a map valued in $M^{p,q}$.		
\par\bigskip	

{\it (c)} We	shall now  prove the desired decay of the solution. By Theorem \ref{mainthm} we have, for $t\geq0$,
		\begin{eqnarray*}
			\|S(t)f \|_{M^{p,q}}\leq C_0 \, e^{-td^\beta} \|f\|_{M^{p,q}}
		\end{eqnarray*}
		for some $C_0>0$. 
		Thus, we have 
		\begin{eqnarray} \label{15bis}
		\|S(t)f\|_{X} \leq C_0 \,  \|f\|_{M^{p,q}}.
		\end{eqnarray}
		We already know from part {\it (a)} that 
		\begin{equation*}\label{eq6}
		 \hspace{-.5in} e^{td^{\beta}}\left\|\int_0^t S(t-\tau) \, [|u(\tau)|^{2k} u(\tau)] \, d \tau \right\|_{M^{p,q}} \nonumber 
		 \leq  C_1 e^{td^{\beta}} \int_0^tC(t-\tau) \| u(\tau)\|^{2k+1}_{M^{p,q}} d\tau
				\end{equation*}
for some $C_1>0$. To control the above integral we divide it  into two parts.  For $t\geq 0$		we		let $E_1= \{ \tau\in [0,t]: t-\tau <1\}$ and $E_2= \{ \tau\in [0,t]: t-\tau \geq 1\}$. Note that 
		\[e^{\tau d^\beta} \|u(\tau)\|_{M^{p,q}}^{2k+1}= e^{-2k\tau d^\beta} (e^{\tau d^\beta}\|u(\tau)\|_{M^{p,q}})^{2k+1}.\] Hence, for some $C_2>0$,
\begin{eqnarray*}
e^{td^{\beta}}\int_{E_1} C(t-\tau) \| u(\tau)\|^{2k+1}_{M^{p,q}} d\tau & \leq & C_2 \, e^{td^{\beta}} \int_{E_1} (t-\tau)^{-\sigma} e^{-\tau d^{\beta}} e^{-2k \tau d^{\beta}} \|u\|_{X}^{2k+1} d\tau\\
& \leq & C_2 \, \|u\|_{X}^{2k+1} e^{d^{\beta}} \int_0^1 s^{-\sigma}  ds 
\end{eqnarray*}
and, since $k>0$,
\begin{eqnarray*}
e^{td^{\beta}}\int_{E_2} C(t-\tau) \| u(\tau)\|^{2k+1}_{M^{p,q}} d\tau & \leq & C_2 \, e^{t d^{\beta}} \int_{E_2} e^{- (t- \tau) d^{\beta}} \|u(\tau)\|_{M^{p,q}}^{2k+1} d\tau\\
& \leq & C_2 \, \|u\|_{X}^{2k+1} \int_{0}^{+\infty} e^{-2k \tau d^\beta} d\tau.
\end{eqnarray*}
Combining these inequalities with \eqref{15bis} yields
		\begin{equation*} \label{eq7}
		\|\mathcal{J}(u)\|_{X} \leq C_3 ( \|u_0\|_{M^{p,q}}+  \|u\|^{2k+1}_{X})
		\end{equation*}
for some constant $C_3>0$. Now, repeating  similar arguments as before gives the desired result.	\par\bigskip
{\it (d)} The proof of the global well-posedness in $W^{q,p}(\mathbb R^d)$ goes as that above for the modulation spaces $M^{p,q}$. We can replace indeed $M^{p,q}$ with $W^{q,p}$ everywhere, using the algebra properties of $W^{q,p}$, analogous to Lemma \ref{Algebra}, that are \[
 \||f|^{2k}f\|_{W^{r,p}}\leq C \|f\|^{2k+1}_{W^{q,p}}
		\]
		for $1\leq p,q,r \leq \infty,  k\in \mathbb N$, $\frac{2k+1}{q}=\frac{1}{r}+2k$,
		which are in turn a consequence of the {\it convolution} properties for modulation spaces \cite{toft contp} (the Fourier transform turns convolution into pointwise multiplication and modulation spaces into Wiener amalgam spaces). 

\end{proof}

\section*{Acknowledgments} \noindent D.\  G.\ B.\ is thankful to Henri Lebesgue Centre (IRMAR, Univ.\ Rennes 1) for the financial support and the excellent research facilities. D.\ G.\ B.\ is also thankful to DST-INSPIRE and TIFR CAM for the academic leave. \\ R.\ M.\ is thankful to Indian Institute of Science (C.V.\ Raman PDF, R(IA)CVR-PDF/2020/224) for the financial support and the excellent research facilities. \\ S.\ T.\ is supported by J.\ C.\ Bose Fellowship from D.\ S.\ T.,
Government of India. \\ F.\ N.\ and S.\ I.\ T.\ are members of the Gruppo Nazionale per l'Analisi Matematica, la Probabilit\`a e le loro Applicazioni (GNAMPA) of the Istituto Nazionale di Alta Matematica (INdAM). The present research was partially supported by MIUR grant “Dipartimenti di Eccellenza” 2018–2022, CUP: E11G18000350001, DISMA, Politecnico di Torino.

\end{document}